\documentclass[a4paper]{amsart}
\usepackage{amsfonts}\usepackage{amssymb}\usepackage{amsmath}\usepackage{amsthm}
\usepackage[utf8]{inputenc}
\usepackage[T1]{fontenc}
\usepackage{latexsym}
\usepackage{graphicx}
\usepackage{layout}
\usepackage{subfigure}
\usepackage{pgfplots}
\usepackage[colorlinks=true,urlcolor=blue,citecolor=red,linkcolor=blue,linktocpage,pdfpagelabels,bookmarksnumbered,bookmarksopen]{hyperref}
\usepackage{tikz}
\usepackage{tikz-3dplot}
\usetikzlibrary{hobby}
\newtheorem{theorem}{Theorem}[section]

\newtheorem{proposition}[theorem]{Proposition}
\newtheorem{lemma}[theorem]{Lemma}
\newtheorem{corollary}[theorem]{Corollary}

\theoremstyle{definition}

\newtheorem{definition}[theorem]{Definition}

\newcommand{\sezione}[1]{\section{#1}\setcounter{equation}{0}}
\numberwithin{equation}{section}

\def\R{\mathbb{R}}
\def\C{\mathbb{C}}

\def\e{{\varepsilon}}
\def\di12{\mathcal{D}^{1,2}(\R^n)}

\def\l{{\lambda}}

\def\0l{_{0,\l}}
\def\1l{_{1,\l}}
\def\2l{_{2,\l}}
\def\3l{_{3,\l}}
\def\4l{_{4,\l}}

\def\Om{\Omega}

\def\beq{\begin{equation}}
\def\eeq{\end{equation}}
\newcommand\esf{\mathbb{S}}
\def\sideremark#1{\ifvmode\leavevmode\fi\vadjust{\vbox to0pt{\vss
 \hbox to 0pt{\hskip\hsize\hskip1em
 \vbox{\hsize2.1cm\tiny\raggedright\pretolerance10000
  \noindent #1\hfill}\hss}\vbox to15pt{\vfil}\vss}}}%

\newtheorem*{theorem*}{Theorem}

\DeclareMathOperator{\mea }{meas}
\DeclareMathOperator{\Sta }{Star}

\begin{document}
\title[Many critical points]{Solutions to general elliptic equations\\ on nearly geodesically convex domains\\ with many critical points}
\author[Enciso]{Alberto Enciso}
\address{Alberto Enciso, Instituto de Ciencias Matemáticas, Calle Nicolás Cabrera 13, 28049 Madrid, Spain e-mail: {\sf aenciso@icmat.es }
}
\author[Gladiali]{Francesca Gladiali}
\address{Francesca Gladiali, Dipartimento SCFMN, Universit\`a  di Sassari  - Via Vienna 2, 07100 Sassari, Italy e-mail: {\sf fgladiali@uniss.it}.}
\author[Grossi]{Massimo Grossi }
\address{Massimo Grossi,  Dipartimento di Scienze di Base Applicate per l’Ingegneria, Universit\`a degli Studi di Roma \emph{La Sapienza} P.le A. Moro 5 - 00185 Roma, e-mail: {\sf massimo.grossi@uniroma1.it}.}



\maketitle
\begin{abstract}
Consider a complete $d$-dimensional Riemannian manifold $(\mathcal M,g)$, a point $p\in\mathcal M$ and a nonlinearity $f(q,u)$ with $f(p,0)>0$. We prove that for any odd integer $N\ge3$, there exists a sequence of smooth domains $\Om_k\subset\mathcal M$ containing $p$ and corresponding positive solutions $u_k:\Om_k\to\R^+$ to the Dirichlet boundary problem
\begin{equation*}
\begin{cases}
                        -\Delta_g u_k=f(\cdot,u_k) \quad  &
            \mbox{  in }\Om_k\,,\\
u_k=0  & \mbox{  on }\partial\Om_k\,.
\end{cases}
\end{equation*}
All these solutions have exactly $N$ critical points in~$\Om_k$ (specifically, $(N+1)/2$ nondegenerate maxima and $(N-1)/2$ nondegenerate saddles), and the domains $\Om_k$ are  star-shaped with respect to~$p$ and become ``nearly convex'', in a precise sense, as  $k\to+\infty$.
\end{abstract}
\sezione{Introduction}\label{s0}
The calculation of the number of critical points of solutions to the equation $-\Delta_g u = f(u)$ has long been a topic of significant interest among mathematicians. Although a substantial body of literature exists on this subject, many challenges remain unresolved, and several open problems appear to be far from a definitive solution.

In this paper we consider the problem 
\begin{equation}\label{T}
\begin{cases}
                        -\Delta_g u=f(q,u)  &
            \mbox{  in }\Omega\\
u=0  & \mbox{  on }\partial\Omega
\end{cases}
\end{equation}
where $\Omega\subset \mathcal M$ is a bounded smooth domain, $\mathcal M$ is a complete $d$-dimensional   Riemannian manifold, possibly non compact,  and $\Delta_g$ is the Laplace--Beltrami operator on $\mathcal M$. Here $f(q,u)$ is a Lipschitz function on $\mathcal M\times\R$.

Despite its seemingly simple structure, this problem is surprisingly rich in intriguing properties. It is well known that the number of critical points of \eqref{T} is strongly influenced by both the topology and the geometry of $\Omega$ (see, for instance, the classic Poincar\'e--Hopf Theorem).

Let us begin by providing a brief historical overview of the problem, starting with the flat case, that is,  $M = \mathbb{R}^d$ with $d\ge2$, equipped with its natural metric. In this context the most general result is undoubtedly the breakthrough theorem of Gidas, Ni, and Nirenberg. In the influential paper \cite{gnn}, it was proven that if $\Omega$ is a domain which is convex with respect to the directions $x_1,..,x_d$, symmetric with respect to the axes $x_1,..,x_d$,
and if the nonlinearity is autonomous (i.e., $f(x,u)\equiv F(u)$), then the solution has a unique critical point located at the center of symmetry of $\Omega$.

Notably, the result in \cite{gnn} holds for general nonlinearities. Several researchers have attempted to extend the theorem of Gidas, Ni, and Nirenberg by removing the symmetry assumption. While it may seem natural that the uniqueness of the critical point should not depend on the symmetry of the domain, this remains a challenging open problem. In the case of the torsion problem ($f(x,u)\equiv1$), a significant result due to Makar--Limanov~\cite{ml} shows that if $\Om\subset\R^2$ is a convex domain then the function $\sqrt{u}$ is strictly concave, which implies the uniqueness of the critical point of the solution $u$. This result was extended to higher dimensions in \cite{KoLe}.  We also highlight the results in \cite{cc} and \cite{dgm}, which consider semi-stable solutions on convex domains of $\R^2$ and proved the uniqueness of the critical point.

It is important to note that the convexity condition on the domain cannot be easily relaxed. Indeed, in \cite{gg4} it was proven that there exists a class of domains $\Om_\e$ which are close to a convex set (in a suitable sense) such that the corresponding solutions of the torsion problem exhibit multiple critical points. 

When one considers analogous problems on Riemannian manifolds $(\mathcal{M},g)$, the results are rather limited (in fact, they mainly focus on the first eigenfunction).
Let us start with the classic case of the surfaces, basically with the model cases  $\esf^2$ (round sphere) and $\mathbb H^2$  (hyperbolic plane), endowed with their natural metrics.
If $\Omega=\esf^2$ and $f\equiv1$ , the uniqueness of the critical point for solutions of \eqref{T}, under the conditions that $\partial\Om$ has positive curvature and the diameter of~$\Om$ is less than $\frac\pi2$, was established in the recent paper \cite{gp}. 

Concerning other nonlinearities, as mentioned earlier, much attention has been given to the case of the eigenfunctions.  In \cite{swyy}  the authors discuss some $\log$-concavity estimates for the first eigenfunction on convex Euclidean domains, providing a short proof of the result of Brascamp and Lieb \cite{bl}. This approach was later adapted in \cite{lw} to prove the $\log$-concavity of the first Dirichlet eigenfunction in geodesically convex domains of $\mathbb \esf^n$. This last result, together with \cite{w}, allows to prove the {uniqueness and non-degeneracy} of the critical point. For related results on positively curved surfaces, see also \cite{khan}.

In the case of hyperbolic space $\mathbb H^2$, the situation is more complex. In fact, the assumption of strict convexity alone does not generally  suffice to ensure the uniqueness of the critical point, as shown for the first eigenfunction in \cite{bcnsww}.  However, whenever $\Om\subset\mathbb H^2$, stronger assumptions such as horoconvexity~\cite{gp} can restore the uniqueness of the critical point for solutions to the torsion problem.

For general Riemannian manifolds,  it is well known that the number of critical points of solutions strongly depend on the metric~$g$. For example, in \cite{eps} it is proved that, given a compact manifold $M$ of dimension $d\geq 3$ and integers~$N,l$, there exits a metric $g$ such that  the lowest~$l$ eigenfunctions of the Laplacian have at least $N$ non-degenerate critical points. Similar results have been obtained in \cite{bls,epss,ms}. However, it is reasonable to expect that, under suitable assumptions on the geodesic convexity, on the curvature of $\mathcal{M}$, and on the nonlinearity~$f$, the uniqueness of the critical point of the solution to \eqref{T} can be established. When the manifold has constant sectional curvature, results of this kind are of course well known.

Our objective in this paper is to show that, as soon as the convexity assumption is slightly relaxed, one cannot hope to control number of critical points of solutions. Furthermore, this phenomenon is independent of the nonlinearity and of the choice of the metric.

Specifically, our main result, which builds on  \cite{gg4}, ensures that, given any nonlinearity~$f$, on any Riemannian manifold, and around any point on the manifold, it is possible to construct a class of domains $\Om_k$ that are ``almost geodesically convex'' as $k\to+\infty$, in the precise sense introduced in Definition \ref{alcon}, where the boundary problem \eqref{T} admits a positive solution with arbitrarily many critical points. In the statement, $\mea_g(\Om)$ denotes the Riemannian measure of the set~$\Om\subset\mathcal M$. We recall that $q$ is a {\em critical point}\/ of a function~$u$ if $\nabla u(q)=0$, and that this critical point is {\em nondegenerate}\/ if its Hessian matrix $\nabla^2 u(q)$, computed in any local chart, is invertible.

\begin{theorem}\label{T1}
Let $(\mathcal M,g)$ be a smooth $d$-dimensional  complete  Riemannian manifold, possibly non compact.  Assume that $ f(q,u)$ and $\frac{\partial  f}{\partial u}(q,u)$ are Lipschitz in a neighborhood of $(p,0)$,  and that $f(p,0)>0$.
For any fixed integer $n\geq 2$ and any point $p\in \mathcal M$ there exists a sequence $\widetilde \Omega_k\ni p$ of smooth bounded domains in $\mathcal M$, with $k\geq1$, such that:
\begin{itemize}
\item[a)] For every $k$, there exists a solution $\widetilde u_k$ in $\widetilde\Omega_k$ to \eqref{T}, that has  exactly $2n-1$  critical points: $n$ nondegenerate maxima and $n-1$ nondegenerate saddles. 
\item[b)] The domain $\widetilde\Omega_k$ is geodesically starshaped with respect to $p$. That is, there exists a starshaped region $U_k\subset T_p\mathcal M$ such that the exponential map $\exp_p$ maps $U_k$ diffeomorphically onto $\widetilde\Omega_k$.
\item[c)] The sequence of sets $\widetilde\Omega_k$ is almost geodesically convex in the sense that
\begin{equation}\label{ge}
\lim_{k\to\infty}\frac{\mea_g\big(\widetilde\Omega_k\setminus \Sta_g (\widetilde\Omega_k)\big)}{\mea_g (\widetilde\Omega_k)}=0\,.
\end{equation}
Here $\Sta_g (\widetilde\Omega_k)$ denotes the set of points of  $\widetilde\Omega_k$ with respect to which $\widetilde\Omega_k$ is geodesically starshaped (see Definition \ref{star-D}).
\end{itemize}
\end{theorem}
The idea of the proof of Theorem \ref{T1} is the following. We start with a domain $\Om_\e\subset\R^d$ and with the solution $u_\e$ to the torsion problem
\begin{equation}\label{tor}
\begin{cases}
                        -\Delta u=1  &
            \mbox{  in }\Omega_\e\,,\\
u=0  & \mbox{  on }\partial\Omega_\e\,.
\end{cases}
\end{equation}
By~\cite{gg4}, we can take a domain  $\Om_\e$ which is ``nearly convex'' as $\e\to0$, and for which $u_\e$ has $n\ge2$ nondegenerate maxima and $n-1$ nondegenerate saddles (and no other critical points). Now we regard $0<\e\ll1$ as small but fixed.  Set $\Phi_\eta(x):=\exp_p(\eta x)$,
where $\exp_p$ is the Riemannian exponential map at~$p$. The ideas is to take $\widetilde\Omega_\eta\subset\mathcal M$ as the image of $\Om_{\e}$ under $\Phi_\eta$, for suitably small $\eta,\e$ depending on~$k$, and  to construct~$\widetilde u_k$ as a small perturbation of  $\widetilde u_\eta:=\eta^2 u_\e\big(\Phi_\eta^{-1}(-)\big):\widetilde\Omega_\eta\to\R^+$. 

At this point,  two serious problems arise:
\begin{itemize}
\item To prove that $\widetilde\Omega_\eta$ satisfies the properties in $a)$, $b)$ and $c)$ in Theorem \ref{T1}.
\item To find a suitably small correction of $\widetilde u_\eta$ yielding a solution to \eqref{T}.
\end{itemize}
To overcome these difficulties, we will need to prove several  delicate additional estimates for the various objects appearing in the construction.

The paper is organized as follows: in Sections \ref{s2} and \ref{s3} we establish some additional properties on the domain $\Om_\e$ which are not covered in~\cite{gg4}. In Section \ref{s4} we introduce a suitably small scale parameter~$\eta$ and analyze the pair $(\widetilde\Om_\eta,\widetilde u_\eta)$. Finally in Section \ref{s5}, we prove Theorem \ref{T1}.

\sezione{Results on $\R^d$}\label{s2}
The proof of Theorem \ref{T1} is based on some results for the torsion problem in the Euclidean space $\R^d$ for $d\geq 2$,  in \cite{gg4} (case $d=2$) and \cite{dg2} (case $d\geq3$). Here we give a slightly different statement of these previous results, because we will need some more properties of the solution $u$ of the torsion problem in $\Omega\subset \R^d$ to prove Theorem \ref{T1}. 
Nevertheless the main ideas of this section can be found in \cite{gg4}.\\

For $d\geq 2$ we consider the torsion problem 
\begin{equation}\label{1}
\begin{cases}
                        -\Delta u=1  &
            \mbox{  on }\Omega\\
u=0  & \mbox{  on }\partial \Omega
\end{cases}
\end{equation}
where $\Omega$ is a smooth bounded domain of $\R^d$.
For $n\ge2$ let us consider arbitrary real numbers
\beq \label{points-a}
0<a_1<a_2<..<a_n\in(0,+\infty)
\eeq 
and the holomorphic function $F:\C\to\C$ given by
\begin{equation}\label{b1}
F(z)=-\Pi_{i=1}^n(z-a_i)(z+a_i)=-\sum_{i=0}^nb_iz^{2i}
\end{equation}
where of course $b_n=1$ and the coefficients $b_i$ depend only on the points $a_i$ in \eqref{points-a}.   

Next, we consider the function $v_n:\R^2\to\R$ defined as 
\begin{equation}\label{eq:v-re}
v_n(x_1,x_2)=Re\Big(F(z)\Big)
\end{equation}
for $z=x_1+ix_2$,
(where $Re$ stands for the real part of a complex function),
which is harmonic. By construction
\begin{equation}\label{eq:v-pj}
v_n(x_1,x_2)=-\sum _{j=0}^n b_{j}P_{2j}(x_1,x_2)
\end{equation}
where $b_{n}=1$ and $P_{2j}$ are some homogeneous harmonic polynomials of degree $2j$ (in which, for symmetry reasons,  $x_1$ and $x_2$ appear only with even degree).  
It might be useful to write
\beq\label{funct-vn}
 v_n(x_1,x_2)=-x_1^{2n}+\tilde b_0x_1^{2n-2}x_2^2+\tilde b_1 x_1^{2n-2}+Q(x_1,x_2)
\eeq
where $Q(x_1,x_2)$ is a polynomial of degree at most $2n-4$ in $x_1$ and $\tilde b_0=n(2n-1)$.

Finally we introduce the function $u_\e:\R^d\to\R$ as
\beq \label{sol}
u_\e(x_1,\dots,x_d):=\frac 1{2(d-1)} (1-x_2^2-\dots-x_d^2)+\e v_n(x_1,x_2)
\eeq
which satisfies the torsion problem 
\[-\Delta u_\e=1\ \ \text{ in }\R^d.\]
Observe that, by construction,  $u_\e$ is even in all the variables.

Let us denote by $f:\R\to\R$ the function
\begin{equation}\label{f}
f(x_1):=v_n(x_1,0).
\end{equation}
We immediately get that $f$ is even, $f(\pm a_1)=..=f(\pm a_n)=0$ and $\pm a_i$ are all simple zeroes for $f$.  Moreover $\lim_{x_1\to \infty}f(x_1)=-\infty$ and $f$ has $n$ nondegenerate maximum points and $n-1$ nondegenerate minimum points in $[-a_n,a_n]$.  The fact that they are all nondegenerate follows since $f(x_1)$ is a polynomial of degree $2n$ with $2n$ simple zeroes and since $f''(x_1)$ has exactly $2n-2$ zeros between the zero of $f'(x_1)$.\\

 Before  stating our main result on the function $u_\e$, we need the following definition:
\begin{figure}[t]
\centerline {\includegraphics[scale=0.3]{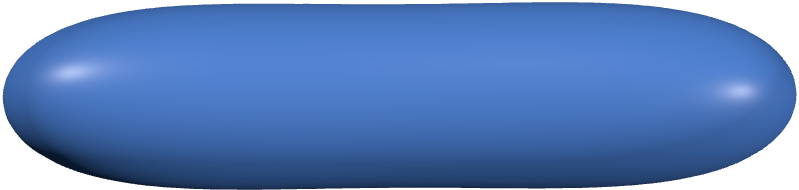}}
\vskip-0.2cm
\caption{The case $d=3$, $n=3$ and $\e=0.01$}
\end{figure}
\begin{definition} \label{star-D}
For any $\Omega\subset\R^d$, we denote by $\Sta (\Omega)$ the set of points $x\in \Omega$ with respect to which $\Omega$ is starshaped.
\end{definition}
In the following we will denote by $B^{d-1}(0,r)$ the ball of radius $r$ in $\R^{d-1}$ centered at $0$  and by $C$ various positive constants independent of $\e$.
The main result of this section is the following:
\begin{theorem}\label{T2}
Let $n\geq 2$. There exists $\e_0>0$ such that for any $0<\e<\e_0$ there exists a smooth bounded domain $\Omega_\e\subset \R^d$ such that 
\begin{itemize}
\item[i)] $u_\e$ is the unique solution to \eqref{1} in $\Omega_\e$ and $\|u_\e\|_{C^l(\Omega_\e)}\leq C$ for all $l$. 
Moreover
$\Omega_\e\subset\left(-D\e^{-\frac 1{2n}},D\e^{-\frac 1{2n}}\right)\times B^{d-1}(0,\sqrt 2)$ for $D=\left(\frac 1{d-1}\right)^{\frac 1{2n}}$.
\item[ii)]  $\Omega_\e$ is smooth and starshaped with respect to the origin.
\item[iii)]  The set $\{(x_1,..,x_d)\in\Om_\e: u_\e(x_1,..,x_d)>\frac 1{2(d-1)}\}$ has $n$ connected components. In each of these components, 
\begin{equation}\label{stmax}
\max u_\e(x_1,..,x_d)>\frac 1{2(d-1)} +C\e.
\end{equation}
\item[iv)]The solution $u_\e$ has exactly $n$ nondegenerate local maximum points and $n-1$ nondegenerate saddle points in $\Omega_\e$.
\item[v)] $\Om_\e$ is ``almost convex'' in the sense that its Lebesgue measure satisfies
$$\mea \big(\Omega_\e\setminus \Sta (\Omega_\e)\big)\to0\quad\hbox{as }\e\to0.$$ 
\item[vi)] If $F$ is the cylinder $F:=  \{(x_1,x_2,..,x_d)\in \R^d \text{ such that } (x_2,..,x_d)\in B^{d-1}(0,1)\}$ and $Q$ is any compact set of $\R^d$ then $\Omega_\e\cap Q\to F\cap Q$ as $\e\to 0$.
\end{itemize}
\end{theorem}
Here, we prove {\rm i)},  {\rm ii)},  {\rm iii)}, {\rm iv)} and {\rm vi)}, while {\rm v)} will be proved in the next section.
\begin{proof}
The proof uses some ideas of Theorem 3.1 in \cite{gg4}.
\vskip0.2cm
{\em Proof of {\rm i)}.}
Let us define 
\beq\label{omega-e}
\Omega_\e=:\{\hbox{the connected component of}: u_\e(x_1,\dots,x_d)>0\hbox{ containing } 0\}.
\eeq
First we prove that $\Omega_\e$ is non empty and contains the line $(t,0,..,0)$ for $t\in (-a_n,a_n)$. 
For $\e$ small we have that, 
\[
\begin{split}
u_\e(t,0,..,0)=\frac 1{2(d-1)}+\e v_n(t,0)\geq \frac 1{2(d-1)}-c_1\e>0
\end{split}\]
where $-c_1=\inf\limits_{t\in (-a_n,a_n)}v_n(t,0)=\inf\limits_{t\in (-a_n,a_n)}f(t)<0$.

Next 
we prove that  $\Omega_\e\subset\left(-D\e^{-\frac 1{2n}},D\e^{-\frac 1{2n}}\right)\times B^{d-1}(0,\sqrt 2)$ when $\e$ is small enough. Let us compute $v_n(\pm D\e^{-\frac 1{2n}},x_2)$ for $|x_2|\leq \sqrt 2$. By \eqref{funct-vn}
\[v_n(\pm D\e^{-\frac 1{2n}},x_2)=-D^{2n}\e^{-1}+O(\e^{\frac{-n+1}n})
\]
so that, recalling that $D=\left(\frac 1{d-1}\right)^{\frac 1{2n}}$ and $n\geq 2$
\[
\begin{split}
u_\e(\pm D\e^{-\frac 3{4n}},x_2,..,x_d)&\leq \frac 1{2(d-1)}-D^{2n}+O(\e^{\frac1n})<0
\end{split}\]
if $\e$ is small enough. 
Moreover, by \eqref{funct-vn}, we have 
\[\sup_{(x_1,x_2)\in \R\times [-\sqrt 2,\sqrt 2] }v_n(x_1,x_2)=C>0\]
and hence if $(x_2,..,x_d)\in \partial B^{d-1}(0,\sqrt 2)$ and $x_1\in \left(-D\e^{-\frac1{2n}},D\e^{-\frac1{2n}}\right)$ then 
\[u_\e(x_1,x_2,..,x_d)\leq -\frac 1{2(d-1)}+C\e<0
\]
if $\e$ is small enough which proves the claim.

Reasoning exactly in the same way, using that $\Omega_\e\subset\left(-D\e^{-\frac 1{2n}},D\e^{-\frac1{2n}}\right)\times B^{d-1}(0,\sqrt 2)$, it can be easily proved that, for $x\in \Omega_\e$ 
\[0\leq u_\e(x_1,x_2,..,x_d)\leq  \frac 1{2(d-1)}+C\e\]
so that $\|u_\e\|_{C^0(\Omega_\e)}\leq C$ and, in a very similar manner, also $\|u_\e\|_{C^l(\Omega_\e)}\leq C$ for every $l$. This concludes the proof of {\rm i)}.

{\em Proof of {\rm ii)}.}
To prove the star-shapedness of $\Omega_\e$ with respect to the origin we will show that
\[x\cdot \nabla u_\e(x)\leq -\frac 1{2(d-1)}<0\]
for every $x\in \partial \Omega_\e$. Note that this also proves the smoothness of $\Omega_\e$.\\
Since $u_\e=0$ on $\partial\Om_\e$ we have that, for $x\in \partial \Omega_\e$,
\beq\label{frontiera}
\frac 1{2(d-1)}+v_n(x_1,x_2)=\frac 1{2(d-1)}(x_2^2+...+x_d^2).
\eeq
Differentiating $u_\e$, we have
\beq\label{derivate-prime}
\frac {\partial u_\e}{\partial x_i}=
\begin{cases}
\e\frac {\partial v_n}{\partial x_1}&if\ i=1\\
-\frac 1{(d-1)}x_2+\e\frac {\partial v_n}{\partial x_2}&if\ i=2\\
-\frac 1{(d-1)}x_i&if\ i>2,
\end{cases}
\eeq
so that
\[
x\cdot \nabla u_\e(x)=\e\left[ \frac {\partial v_n}{\partial x_1}x_1+ \frac {\partial v_n}{\partial x_2}x_2\right]-\frac 1{(d-1)}\sum_{i=2}^d x_i^2
\]
and using \eqref{frontiera},
\[\begin{split}
x\cdot \nabla u_\e(x)&=-\frac 1{(d-1)}+\e\left[ \frac {\partial v_n}{\partial x_1}x_1+ \frac {\partial v_n}{\partial x_2}x_2-2v_n(x_1,x_2)\right].
\end{split}
\]

Since
\[v_n(x_1,x_2)=-\sum_{j=0}^{n}b_{j}P_{2j}(x_1,x_2)\]
with $b_{n}=1$, by Euler's theorem, for homogeneous function, we get
\[\begin{split}
&\frac {\partial v_n}{\partial x_1}x_1+ \frac {\partial v_n}{\partial x_2}x_2-2v_n(x_1,x_2)=-\sum_{j=0}^n(2j-2)b_{j}P_{2j}(x_1,x_2).
\end{split}
\]
Observe that
\[\lim_{|x_1|\to +\infty}-\sum_{j=0}^{n}b_{j}P_{2j}(x_1,x_2)=-\infty\]
which implies 
\[\sup _{(x_1,x_2)\in \R\times (-\sqrt 2,\sqrt 2)}-\sum_{j=0}^{n}(2j-2)b_{j}P_{2j}(x_1,x_2)=l<\infty\]
and finally 
\[\begin{split}
x\cdot \nabla u_\e(x)&\leq- \frac 1{(d-1)}+l\e\leq -\frac 1{2(d-1)}<0
\end{split}
\]
for $\e$ small enough.

{\em Proof of {\rm iii)}.}
Here we show that the set $\Omega_\e':=\{x\in \Omega_\e: u(x_1,\dots,x_d)>\frac 1{2(d-1)}\}$ has exactly  $n$ connected components.
By the definition of $f$ in \eqref{f} there exist $n$ maximum points $\bar s_1,..,\bar s_{n}$ for $f$ in $[-a_n,a_n]$ and  $n-1$ minimum points $s_1,..,s_{n-1}$ for $f$ in $[-a_n,a_n]$ such that
\[ -a_n<\bar s_1<s_1<\bar s_2<s_2<...<s_{n-1}<\bar s_{n}<a_n\]
and
\[\min\limits_{1\le j\le n}f(\bar s_j)=d'>0 \ , \ \max\limits_{1\le j\le n-1} f(s_j)=-d''<0.
\]
Moreover  we observe that 
\beq\label{stue}
\begin{split}
u_\e(\bar s_j,0,..,0)
&=\frac 1{2(d-1)}+\e v_n(\bar s_j,0)=\frac 1{2(d-1)}+\e f(\bar s_j)\\
&\geq \frac 1{2(d-1)}+\e d'>\frac 1{2(d-1)},\\
\end{split}
\eeq
which proves \eqref{stmax} with $C=d'$.

Next we want to prove that 
\beq\label{stima-components-1}
u(b,x_2,..,x_d)<\frac 1{2(d-1)}
\eeq
when either $b=s_j$ for $j=1,\dots,n-1$ or $b=\pm 2a_n$ and $(x_2,..,x_d)\in B^{d-1}(0,\sqrt 2)$.\\
To prove \eqref{stima-components-1} we argue by contradiction and we assume there exist a sequence $\e_k\to 0$ and points $(x_{2,k},..,x_{d,k})\in B^{d-1}(0,\sqrt 2)$
 such that 
\beq \label{stima-inv}
\begin{split}
u(b,x_{2,k},..,x_{d,k})&=\frac 1{2(d-1)}(1-x_{2,k}^2-..-x_{d,k}^2)+\e_k
v_n(b,x_{2,k})\geq \frac 1{2(d-1)}
\end{split}
\eeq
for $k\to \infty$, for a fixed value of $b$. Since $\e_k v_n(b,x_{2,k})\to 0$ as $k\to \infty$, formula \eqref{stima-inv} easily implies that $(x_{2,k},..,x_{d,k})\to 0$ as $k\to \infty$.\\

Next we observe that $v_n(b,0)=f(b)\leq -\beta:=\max\{-d'', f(\pm 2 a_n)\}<0$ so that $v_n(b,x_{2,k})\leq -\frac 12\beta$ for $k$ large enough.  Moreover, 
\beq \label{stima-inv-2}
\begin{split}
u(b,x_{2,k},..,x_{d,k})&=\frac 1{2(d-1)}(1-x_{2,k}^2-..-x_{d,k}^2)+\e_k v_n(b,x_{2,k})\\
&\leq \frac 1{2(d-1)}-\frac 12 \e_k\beta<\frac 1{2(d-1)}
\end{split}
\eeq
for $\e_k$ small enough. This contradicts \eqref{stima-inv} and proves \eqref{stima-components-1}.

Estimates \eqref{stue} and \eqref{stima-components-1} imply that the set $\{(x_1,..,x_d)\in\Om_\e: u_\e(x_1,..,x_d)>\frac 1{2(d-1)}\}$ has at least  $n$ connected components, the first one in the subset $\Omega_\e\cap [-2a_n, s_{1}]\times B^{d-1}(0,\sqrt 2)$, the last one in the subset $\Omega_\e\cap [s_{n-1},2a_n]\times B^{d-1}(0,\sqrt 2)$, and the others in the subsets $\Omega_\e\cap [s_j,s_{j+1}]\times B^{d-1}(0,\sqrt 2)$ for $j=1,..,n-2$.  In each of these components $u_\e$ admits at least one maximum point.

{\em Proof of {\rm iv)}:} we have to show that $u_\e$ has exactly $n$ maximum points, $n-1$ saddle points, and that they are all nondegenerate. Note $u_\e$ cannot have any minima by the maximum principle.

The critical points of $u_\e$ must satisfy
\beq\label{derivate-prime=0}
\begin{cases}
\frac {\partial v_n}{\partial x_1}=0&\\
-\frac 1{(d-1)}x_2+\e\frac {\partial v_n}{\partial x_2}=0&\\
-\frac 1{(d-1)}x_i=0&\hbox{for }i\ge3
\end{cases}
\eeq
so that $x_3=..=x_d=0$ at a critical point.

Moreover the first equation in \eqref{derivate-prime=0} implies that if $P_\e=(x_1^\e,x_2^\e,0,..,0)$ is a critical points of $u_\e$ in $\Omega_\e$, then $(x_1^\e,x_2^\e)$ is a critical point of $v_n$ contained in $(-D\e^{-\frac 1{2n}}, D\e^{-\frac 1{2n}})\times (-\sqrt 2, \sqrt 2)$. Assume, by contradiction, that $|x_1^\e|\to +\infty$ as $\e \to 0$. Then $\frac {\partial v_n}{\partial x_1}(x_1^\e,x_2^\e)\to \pm \infty$ as $\e \to 0$, since $\frac {\partial v_n}{\partial x_1}(x_1,x_2)\sim -2n x_1^{2n-1}$ as $|x_1|\to \infty$.  This is not possibile since we are assuming $\frac {\partial v_n}{\partial x_1}(x_1^\e,x_2^\e)=0$. This proves that $|P_\e|\leq C$ for a suitable constant $C$.

Now we  consider a maximum point  $P_\e^j=(x_1^\e,x_2^\e,0,..0)\in [s_j,s_{j+1}]\times B^{d-1}(0,\sqrt 2)$.  The case when either $P_\e\in [-2a_n,s_{1}]\times B^{d-1}(0,\sqrt 2)$ or  $P_\e\in [s_{n-1},2a_n]\times B^{d-1}(0,\sqrt 2)$ follows exactly in the same way. 
The second equation in \eqref{derivate-prime=0} gives
\beq
\label{comp-sec}
\frac 1{(d-1)}x_2^\e=\e\frac {\partial v_n}{\partial x_2}(x_1^\e,x_2^\e)= O(\e)
\eeq
and so $x_2^\e=O(\e)$ as $\e\to 0$. 

Inserting in the first equation of \eqref{derivate-prime=0} we obtain 
\beq\label{comp-prima}
0=\frac {\partial v_n}{\partial x_1}(x_1^\e,x_2^\e)= \frac {\partial v_n}{\partial x_1}(x_1^\e,0)+O(x_2^\e)=f'(x_1^\e)+O(\e)\eeq
and this implies that $x_1^\e\to\bar x_j$ as $\e\to 0$ and $\bar x_j$ must satisfy  $f'(\bar x_j)=0$. 

Since $P_\e^j$ is a maximum point and $\bar x_j\in [s_j,s_{j+1}]$ then $\bar x_j=\bar s_j$ and $x_1^\e\to \bar s_j$. 
The Hessian matrix of $u_\e$  in $P_\e^j$ is 
\beq\label{hessian}
H_{u_\e}=
\begin{pmatrix}
A_\e & 0\\
0&B
\end{pmatrix}\eeq
where $B=-\frac 1{d-1}I$ and $I$ is the identity matrix,    while
\[A_\e=
\begin{pmatrix}
\e \frac {\partial^2 v_n}{\partial x_1^2}(x_1^\e,x_2^\e)& \e \frac {\partial^2 v_n}{\partial x_1\partial x_2}(x_1^\e,x_2^\e)\\
\e\frac {\partial^2 v_n}{\partial x_1\partial x_2}(x_1^\e,x_2^\e)&-\frac 1{(d-1)}+\e\frac {\partial^2 v_n}{\partial x_2^2}(x_1^\e,x_2^\e)
\end{pmatrix}\]

Let us observe that 
\[
\begin{split}
\frac {\partial^2 v_n}{\partial x_1^2}(x_1^\e,x_2^\e)&=\frac {\partial^2 v_n}{\partial x_1^2}(x_1^\e,0)+O(x_2^\e)=\frac {\partial^2 v_n}{\partial x_1^2}(\bar s_j,0)+O(|x_1^\e-\bar s_j|)+O(\e)\\
&=f''(\bar s_j)+O(|x_1^\e-\bar s_j|)+O(\e)=f''(\bar s_j)+o(1)
\end{split}
\]
where $f''(\bar s_j)<0$, and also that $\frac {\partial^2 v_n}{\partial x_i\partial x_j}(x_1^\e,x_2^\e)=O(1)
$. Concluding we have 
\[\begin{split}
&{\rm det} A_\e=-\frac 1{(d-1)}\e f''(\bar s_j)(1+o(1)>0\quad\hbox{and }{\rm tr} A_\e<0
\end{split}\]
for $\e$ small enough. So the Hessian matrix $H_{u_\e}$ is negative definite at $P_\e^j$ which is a nondegenerate maximum point for $u_\e$.
This implies also that $u_\e$ has a unique maximum point in $\Omega_\e\cap [s_j,s_{j+1}]\times B^{d-1}(0,\sqrt 2)$ for $j=1,..,n-2$ and also in 
$\Omega_\e\cap [-2a_n,s_{1}]\times B^{d-1}(0,\sqrt 2)$ or  $\Omega_\e\cap [s_{n-1},2a_n]\times B^{d-1}(0,\sqrt 2)$.

Next we let $P_\e=(x_1^\e,x_2^\e,0,..0)$ be another critical point for $u_\e$ in $\Omega_\e$ different from the previous maximum points $P_\e^j$.  Then $|P_\e|<C$. We can then reason exactly as above and get that $x_2^\e\to 0$ as $\e\to 0$ and that $x_1^\e\to \tilde x_1\in [-C,C]$ where $\tilde x_1$  satisfies $f'(\tilde x_1)=0$.  By the properties of the function $f$ then either $\tilde x_1=\bar s_j$ for $j=1,..,n$ or $\tilde x_1=s_j$ for $j=1,..,n-1$.  But $\tilde x_1=\bar s_j$ is not possibile since $f''(\bar s_j)< 0$ and $\bar s_j$ is the limit of the maximum point $P_\e^j$.  Then $\tilde x_1=s_j$ and $f''(s_j)>0$. 

Computing the Hessian matrix of $u_\e$ at the point $P_\e$ and reasoning as above we have that the submatrix $B$ is negative definite while the submatrix $A_\e$ satisfies 
${\rm det} A_\e=-\frac 1{(d-1)}\e f''(s_j)(1+o(1))<0$
for $\e$ small and this implies that for $\e$ small the Hessian matrix $H_{u_\e}$ in $P_\e$ has one positive eigenvalue.  This  shows in turn that $P_\e$ is a nondegenerate saddle point for $u_\e$. 
Then $u_\e$ in $\Omega_\e$ has exactly $n$ nondegenerate maximum points and can have only nondegenerate saddle points.  The nondegeneracy of all the critical points of $u_\e$ in $\Omega_\e$ implies that they are isolated so that we can apply the Poincaré-Hopf Theorem to the vector field $\nabla u_\e$ in $\Omega_\e$ (observe that $\nabla u_\e\cdot \nu<0$ on $\partial \Omega_\e$) and we get that
\[ \sum_{P_\e} {\rm index}_{P_\e}\left(\nabla u_\e\right)=(-1)^d \chi(\Omega_\e)\]
where $\chi(\Omega_\e)=1$ is the Euler characteristic of $\Omega_\e$. This means that the saddle points of $u_\e$ are exactly $n-1$ and concludes the proof of {\rm iv)}. 

{\em Proof of {\rm vi)}.} 
The property {\rm vi)} easily follows by the definition of $u_\e$ since if $x\in Q$ and $Q$ is compact then $\e v_n(x_1,x_2)=O(\e)$ and $u_\e\to \frac 1{2(d-1)} (1-x_2^2-\dots-x_d^2)$. 
\end{proof}
\section{The ``almost convexity'' property of $\Omega_\e$}\label{s3}
Here we prove {\rm v)} of Theorem \ref{T2}.
 We want to show that the set of points $\Sta (\Omega_\e)$ with respect to which $\Omega_\e$ is starshaped  (see definition \ref{star-D}) differs from $\Omega_\e$ by a set whose measure converges to zero as $\e\to 0$.

For this, let us consider the set 
\begin{equation}\label{d-epsilon}
D_\e=\left\{\xi=(\xi_1,\xi_2,..,\xi_d)\in\Omega_\e \Big| u_\e(\xi)>A_0\e^\frac1n\right\}
\end{equation}
for a suitable value of $A_0$.

\begin{proposition} \label{lem1}
There exits a positive constant $\bar A_0 =\bar A_0(n,D,a_i)$ where $D$ is the constant appearing in {\rm i)} of Theorem \ref{T2} and $a_i$ are as in \eqref{points-a} such that for any $A_0>\bar A_0$ we have
\begin{equation}\label{g}
D_\e\subset \Sta (\Omega_\e).
\end{equation}
\end{proposition}
\begin{proof}
Recall that
$u_\e(x_1,..,x_d)=\frac 1{2(d-1)} (1-x_2^2-...-x_d^2)+\e v_n(x_1,x_2)$
and if $x\in \partial \Omega_\e$ then
\beq\label{gfrontiera}
\frac 1{2(d-1)}+\e v_n(x_1,x_2)=\frac 1{2(d-1)}(x_2^2+...+x_n^2)
\eeq
We want to show that $(x-\xi)\cdot \nabla u_\e(x)<0$ for every $\xi\in D_\e$ and for every $x\in \partial \Omega_\e$. Since $\nabla u_\e(x)$ is proportional to the unit normal at~$x\in\partial\Omega_\e$, this will prove that $D_\e\subset \Sta (\Omega_\e)$.

From now we will often write
 $$h(x,\xi)=O\big(g(\e)\big)$$
when $|h(x,\e)|\le Cg(\e)$
for some positive constant $C$ depending only $n,D,a_i$.
 Let us compute
\begin{equation}\label{g12}
\begin{split}
&(x-\xi)\cdot \nabla u_\e(x)=\underbrace{(x_1-\xi_1)\e\frac {\partial v_n}{\partial x_1}}_{:=I_1}+\underbrace{(x_2-\xi_2)\e\frac {\partial v_n}{\partial x_2}}_{:=I_2}
-\underbrace{\frac1{d-1}\sum_{i=2}^d(x_i-\xi_i)x_i}_{:=I_3},
\end{split}
\end{equation}
and estimate all the terms. We have that
\begin{equation}
\begin{split}
\frac{I_1}\e&=(x_1-\xi_1)\frac {\partial v_n}{\partial x_1}=(x_1-\xi_1)\left(-2nx_1^{2n-1}+O(|x_1|^{2n-3})\right)\\
&=-2nx_1^{2n}+2n\xi_1x_1^{2n-1}+O\big((|x_1|+|\xi_1|)|x_1|^{2n-3}\big)\\
&\quad\left(\hbox{using Young's inequality }ab\le\frac1{2n}a^{2n}+\frac{2n-1}{2n}b^\frac{2n}{2n-1}\right)\le\\
&\le -2nx_1^{2n}+\xi_1^{2n}+(2n-1)x_1^{2n}+O\big(|x_1|^{2n-2}+|\xi_1||x_1|^{2n-3}\big)\\
&=v_n(x_1,x_2)-v_n(\xi_1,\xi_2)-\underbrace{v_n(x_1,x_2)-x_1^{2n}}_{=O(x_1^{2n-2})}+
\underbrace{v_n(\xi_1,\xi_2)+\xi_1^{2n}}_{=O(\xi_1^{2n-2})}+O\big(\e^\frac{-n+1}n\big)\\[3mm]
&\Rightarrow I_1=\e\Big(v_n(x_1,x_2)-v_n(\xi_1,\xi_2)\Big)+O\big(\e^\frac1n\big).
\end{split}
\end{equation}
Moreover
\begin{equation}
I_2=(x_2-\xi_2)\e\frac {\partial v_n}{\partial x_2}=O\left(\e|x_1|^{2n-2}\right)=O(\e^\frac1n),
\end{equation}
and
\begin{equation}
\begin{split}
I_3&=-\frac1{d-1}\sum_{i=2}^d(x_i-\xi_i)x_i=-\frac1{d-1}\sum_{i=2}^dx_i^2+\frac1{d-1}\sum_{i=2}^d\xi_ix_i\\
&\le-\frac1{d-1}\sum_{i=2}^dx_i^2+\frac1{2(d-1)}\sum_{i=2}^d\xi_i^2+\frac1{2(d-1)}\sum_{i=2}^dx_i^2\\
&=-\frac1{2(d-1)}\sum_{i=2}^dx_i^2+\frac1{2(d-1)}\sum_{i=2}^d\xi_i^2\,.
\end{split}
\end{equation}
Hence \eqref{g12} becomes
\begin{equation}\label{g99}
\begin{split}
(x-&\xi)\cdot \nabla u_\e(x)=I_1+I_2+I_3\\
&\le-\frac1{2(d-1)}\sum_{i=2}^dx_i^2+\e v_n(x_1,x_2)+
\frac1{2(d-1)}\sum_{i=2}^d\xi_i^2-\e v_n(\xi_1,\xi_2)+O(\e^\frac1n)\\
&\qquad\left(\text{adding and subtracting  }\frac 1{2(d-1)}\right)\\
&=\underbrace{\frac1{2(d-1)}\left(1-\sum_{i=2}^dx_i^2\right)+\e v_n(x_1,x_2)}_{=0\text{ since }x\in\partial\Om_\e}-\underbrace{\frac 1{2(d-1)}\left(1-
\sum_{i=2}^d\xi_i^2\right)-\e v_n(\xi_1,\xi_2)}_{\le -A_0\e^\frac1n\text{ since }\xi\in D_\e}+O(\e^\frac1n)\\
&\le-A_0\e^\frac1n+O(\e^\frac1n)\le -A_0\e^\frac1n+\bar A_0\e^\frac1n.
\end{split}
\end{equation}
where $\bar A_0 =\bar A_0(n,D,a_i)$. Then we can choose in \eqref{g99} $A_0>\bar A_0$ such that $(x-\xi)\cdot \nabla u_\e(x)<0$ for $\e$ small enough. This ends the proof.
\end{proof}

For our future applications, an important observation is that $D_\e$ is stable with respect $C^2$-small deformation of the domain~$\Omega_\e$.
\begin{corollary}\label{C.perturbation}
	Let us fix $\e>0$ and $A_0>\bar A_0$ as in Proposition~\ref{lem1}, and consider the set $D_\e$ in \eqref{g}. Let $S$ be an open set $S\supset \overline\Omega_\e$, then there exists  $\delta_0=\delta_0(\e)>0$, such that $D_\e\subset \Sta(\Phi(\Om_\e))$ for any map $\Phi:S\to\R^d$ satisfying $\|\Phi-I\|_{C^2(S)}< \delta_0$.
\end{corollary}

\begin{proof}
	The proof of Proposition~\ref{lem1} shows that, for any $\xi\in D_\e$ and any $x\in\partial\Om_\e$, $(x-\xi)\cdot \nu_\e(x)\geq c_0>0$, where $\nu_\e(x)$ is the unit normal of~$\partial\Om_\e$. By continuity, this means that
	\[
	(x_*-\xi)\cdot \nu_*(x_*)\ge c_0/2>0
	\]
	for all $x_*\in\partial \Om_*:=\partial \Phi(\Om_\e)$, provided that $\Phi$ is sufficiently close to the identity in the $C^2$-norm. Here $\nu_*$ denotes the unit normal of~$\partial\Om_*$.
\end{proof}
Our final result is to show that the measure of $D_\e$ is close to that of $\Om_\e$.
\begin{lemma} \label{lem2}
Let $\bar A_0 =\bar A_0(n,D,a_i)$ be as in Proposition \ref{lem1} and let  $A_0>\bar A_0$ in $D_\e$ in \eqref{d-epsilon}.
Then
\begin{equation}\label{finale}
\frac{\mea (\Om_\e\setminus D_\e)}{\mea(\Om_\e)}\to0\quad\hbox{as }\e\to0.
\end{equation}
\end{lemma}
\begin{proof}
By the definition of $D_\e$ we have that 
\begin{equation}
\begin{split}
\Om_\e\setminus D_\e&=\{(x_1,..,x_d)\in \Omega_\e
:0\le u_\e(x)\le A_0\e^\frac1n\}.
\end{split}
\end{equation}
With the change of variablse $y_1:\big(2(d-1)\big)^\frac1{2n}\e^\frac1{2n}x_1,$ $ y_i:=x_i \hbox{ for }i\ge2$, we get
\begin{equation}\label{g9}
\mea (\Om_\e\setminus D_\e)=\int_{\Om_\e\setminus D_\e}
dx=
\big(2(d-1)\big)^{-\frac 1{2n}}
{\e^{-\frac1{2n}}}\int_{K_\e}dy.
\end{equation}
Here, by {\rm i)} of Theorem \ref{T2}, and letting $C':=\big(2(d-1)\big)^{-\frac 1{2n}}D$ (where $D$ is the constant in Theorem \ref{T2}), we have set
\[
\begin{split}
K_\e:=&\left\{y\in(-C',C')\times B^{d-1}(0,\sqrt 2) :0\le u_\e\left(\big(2(d-1)\big)^{-\frac1{2n}}\e^{-\frac1{2n}}y_1, y_2,..,y_d\right)\le  A_0\e^\frac1n\right\}=\\
&\! \!\!\!\!=\left\{y\in (-C',C')\times B^{d-1}(0,\sqrt 2):0\le\frac 1{2(d-1)}\left(1-\sum_{i=2}^dy_i^2\right)-\frac 1{2(d-1)}
y_1^{2n}+O(\e^\frac1{n})\le  A_0\e^\frac1n\right\}\\
&\subseteq\underbrace{\left\{y\in (-C',C')\times B^{d-1}(0,\sqrt 2):-B\e^\frac1n\le1-\sum_{i=2}^dy_i^2-
y_1^{2n}\le  A\e^\frac1n\right\}}_{=H_\e}
\end{split}
\]
for some suitable constants $A,B>0$. Thus we have $\mea (K_\e)\leq \mea (H_\e)$.

By the symmetry of $H_\e$ we have 
\[\int_{H_\e}dy=2\int_{H_\e^+}dy\]
where 
\[
H_\e^+=\left\{y\in [0,C')\times B^{d-1}(0,\sqrt 2):1-A\e^\frac1{n}\le \sum_{i=2}^dy_i^2+
y_1^{2n}\le  1+B\e^\frac1{n}\right\}.\]
With the change of variables $z_1:=y_1^n$, $z_i:=y_i$ for $i\geq 2$, we get
\[
\int_{H_\e^+}dy=\frac 1n \int_{Z_\e^+}z_1^{\frac {1-n}n} dz
\]
where 
$Z_\e^+=\{z\in[0,(C')^n)\times B^{d-1}(0,\sqrt 2)
:1-A\e^\frac1n\le \sum_{i=1}^d z_i^{2}\le1+B \e^\frac1n
\}$.

In spherical coordinates we have $z_1=\rho \cos\theta_1$ and 
\beq 
\int_{Z_\e^+}z_1^{\frac {1-n}n} dz=\int_{\esf^{d-1}_+}(\cos\theta_1)^{\frac {1-n}n}w(\theta)d\theta\int_{\sqrt{1-A\e^\frac 1n}}^{\sqrt{1+B\e^\frac1n}} \rho^{d-2+\frac1n}
d \rho 
\eeq
where $\esf^{d-1}_+:=\{(\theta_1,..\theta_{d-1}): \theta_1\in [0,\frac \pi 2) , \theta_i\in [0,2\pi)\}$ and $w(\theta)$ is the surface element, which satisfies $|w(\theta)|\leq C$.  Since 
$$\int_{\esf^{d-1}_+}(\cos\theta_1)^{\frac 1n-1}w(\theta)d\theta\leq E$$
we have 
\beq \begin{split}
&\int_{Z_\e^+}z_1^{\frac {1-n}n} dz\leq E\int_{\sqrt{1-A\e^\frac1n}}^{\sqrt{1+B\e^\frac1n}} \rho^{d-2+\frac{1}n} d \rho \\
& 
=E'\left
[\left(1+B\e^\frac1n\right)^{\frac{d-1+\frac{1}n}2}-\left(1-A\e^\frac1n\right)^{\frac{d-1+\frac{1}n}2}\right]=O(\e^\frac1n)
\end{split}
\eeq
for suitable constants $E,E'$. Together with \eqref{g9}, this gives 
\begin{equation}
\mea (\Om_\e\setminus D_\e)\leq 2
\big(2(d-1)\big)^{-\frac 1{2n}}
{\e^{-\frac1{2n}}\int_{H_\e^+}dy\leq C \e^{\frac1{2n}}}=o(1)
\end{equation}
as $\e\to 0$. The result follows easily from this.
\end{proof}

\begin{proof}[Proof of {\rm v)} in Theorem \ref{T2}.]
The property {\rm v)} in Theorem \ref{T2} follows by Lemma \ref{lem1} and Lemma \ref{lem2}.
\end{proof}
\section{The exponential map}\label{s4}
From now we fix $\e$  such that Theorem \ref{T2} holds. 
Take a point $p\in\mathcal M$ and  let $\exp_p:V\to U$ be the exponential map  at $p$, where $V$ is a neighborhood of the origin in $T_p\mathcal M$ and $U$ is a neighborhood of $p$ in  $\mathcal M$.
 Let 
\beq\label{Phi}
\Phi_\eta(x):=\exp_p(\eta x)
\eeq
for $\eta\neq 0$ and $x\in\frac V\eta\subset T_p\mathcal M$.

We want to embed the domain $\Omega_\e$ of Theorem \ref{T2} into $\mathcal M$ using the map $\Phi_\eta(x)$. 
Corresponding to $\e$ there exists a $\delta_\e>0$, which depends on the manifold $\mathcal{M}$ and on the point $p\in \mathcal M$, such that for every $0<|\eta|<\delta_\e$ we have that
$$\eta x\in V\quad\hbox{for any }x\in\Om_\e.$$
In this way,  for every $0<|\eta|<\delta_\e$,   $\Phi_\eta$ is a diffeomorphism from $\Omega_\e$ into its image that  we will denote  
\beq\label{omt}
\widetilde \Omega_\eta=\Phi_\eta(\Omega_\e)=\exp_p(\eta\Omega_\e).
\eeq

In these coordinates (``rescaled normal coordinates''), the metric can be written as
\beq\label{metric}
g_{jk}=\eta^2 \left(\delta_{jk}+h_{jk}(\eta x)\right)
\eeq
where $h_{jk}$ are smooth functions, because the metric  is smooth. It is standard that 
\begin{align}
|h_{jk}(\eta x)|&\le C\eta^2|x|^2,\\
\label{metric2}
|\nabla_x h_{jk}(\eta x)|&\le C\eta^2|x|
\end{align}
for some positive constant $C$ independent of $\eta$ and $x\in\Om_\e$.

Recall that, in local coordinates, the Laplacian reads as 
\beq
\Delta_g=\frac 1{\sqrt {|g|}}\left(\frac{\partial}{\partial x_j} \sqrt {|g|} g^{jk}\frac{\partial}{\partial x_k}\right)=
g^{jk}\frac{\partial^2}{\partial x_j\partial x_k}+\frac1{\sqrt{g}}
\frac{\partial}{\partial x_j}\left(\sqrt{g}g^{jk}\right)\frac\partial{\partial x_k}\,.
\eeq
Since in our rescaled coordinates
\begin{gather} \label{metric-inv}
|g|=|\det (g_{ij})|=\eta^{2d}\left(1+O(\eta^2|x|^2)\right)\ ,\qquad  \ g^{jk}=\frac1{\eta^2}\left(\delta_{jk}+O(\eta^2|x|^2)\right)\,,\\
\sqrt{g}g^{jk}=\eta^{d-2}\left(\delta_{jk}+O(\eta^2|x|^2)\right),
\end{gather}
one obtains
\beq \label{lapl-belt}
\begin{split}
&\Delta_g=\frac1{\eta^2}\left(\delta_{jk}+O(\eta^2|x|^2)\right)\frac{\partial^2}{\partial x_j\partial x_k}+\frac1{\eta^2}O(\eta^2|x|)\frac\partial{\partial x_k}=\frac 1{\eta^2}\left(\Delta_{\R^d}-\eta^2 T_\eta \right)\,.
\end{split}
\eeq
Here $\Delta_{\R^d}:=\frac{\partial^2}{\partial x_1^2} +\dots+\frac{\partial^2}{\partial x_d^2}$ and $T_\eta$ is an operator of the form
\beq\label{stB}
T_\eta=A_{jk}(\eta,x)\frac{\partial^2}{\partial x_j\partial x_k}+A_k(\eta,x)\frac{\partial}{\partial x_k}
\eeq
with
\beq\label{stA}
A_{jk}(\eta,x)= \frac1{\eta^2}\delta_{jk}-g^{jk}=O(|x|^2)\quad\hbox{and }A_k(\eta,x)=O(|x|)\quad\hbox{for }x\in\Om_\e.
\eeq

Now we can state the main result of this section
\begin{proposition}\label{prog}
Assume $\tilde f(q,u)$, $\frac{\partial \tilde f}{\partial u}(q,u)$ are Lipschitz in a neighborhood of  of $(p,0)$ 
and $\tilde f(p,0)>0$. For every $\e$ fixed, that verifies $0<\e<\e_0$, there exists $\eta_{0,\e}>0$ such that, for every $0<\eta<\eta_{0,\e}$ the problem
\begin{equation}\label{general-f}
\begin{cases}
                        -\Delta_g \tilde u= \tilde f(q,\tilde u)&
            \mbox{  in }\widetilde\Omega_\eta\\
\tilde u=0  & \mbox{  on }\partial\widetilde\Omega_\eta
\end{cases}
\end{equation}
admits a continuous branch of solutions $\tilde u_\eta$ such that
\beq\label{sstim}
\left\|\frac{\tilde u_\eta\big(\exp_p(\eta \cdot)\big)}{\eta^2}-u_\e
\right\|_{C^{2,\alpha}( \overline\Omega_\e)}\to0\quad\hbox{as }\eta\to0,
\eeq
where $u_\e$ is the solution of Theorem \ref{T2}. 
\end{proposition}
\begin{proof}
We prove the result in the case that $\tilde f(p,0)=1$.  The general case
can be reduced to this by observing that, if $\tilde u$ is a solution to \eqref{general-f}, the function $v(x):= {\tilde u(x)} /{\tilde f(p,0)}$
satisfies  $-\Delta_g \tilde v= g(q,\tilde v)$  in $\widetilde\Omega_\eta$, with $g(q,v):=\tilde f(q,\tilde f(p,0)v)/\tilde f(p,0)
$, and  $g(p,0)=1$.

In our rescaled coordinates, $q=\Phi_\eta(x)=\exp_p(\eta x)$. Equation \eqref{general-f} becomes
\begin{equation}\label{general-f-coord}
\begin{cases}
                      \frac 1{\eta^2}\left[ -\Delta _{\R^d}u+\eta^2T_\eta u\right]= f(\eta x, u)  &
            \mbox{  in }\Omega_\e\\
u=0  & \mbox{  on }\partial \Omega_\e
\end{cases}
\end{equation}
where 
$$u=\tilde u\circ \Phi_\eta\quad\hbox{ and } f(\eta x, u):=\tilde f(\Phi_\eta(x),\tilde u).$$ 
Observe that $f(0,0)=\tilde f(p,0)=1$. 
Set 
$$u:=\eta^2 U$$ 
so that \eqref{general-f-coord} becomes
\begin{equation}\label{general-f-coord-U}
\begin{cases}
                       -\Delta _{\R^d}U+\eta^2T_\eta U=  f(\eta x, \eta^2 U) &
            \mbox{  in }\Omega_\e\\
U=0  & \mbox{  on }\partial \Omega_\e.
\end{cases}
\end{equation}

Solving \eqref{general-f-coord-U} produces solutions to \eqref{general-f}. This will be done by using the implicit function theorem.  To this end let us define the function
$$\mathcal{G}=\mathcal{G}:(-\delta_\e,\delta_\e)\times C^{2,\alpha}(\overline\Omega_\e)\cap  C^{0,\alpha}_0 (\overline \Omega_\e)\to C^{0,\alpha}(\overline\Omega_\e)$$
as
\beq \label{g-epsilon}
\mathcal{G}(\eta,U):=\Delta _{\R^d}U-\eta^2T_{\eta} U+ f(\eta x, \eta^2 U)\,.
\eeq 
Note this is well defined if  $0<|\eta|<\delta_\e$, and 
\[
\mathcal G(0,U)=\Delta _{\R^d}U+1\,.
\]

In order to apply the implicit function theorem we need to prove that
\begin{itemize}
\item[i)] $\mathcal{G}(0,u_\e)=0$.
\item[ii)] $\mathcal{G}$ is continuous in a neighborhood of $(0,u_\e)$.
\item[iii)] The derivative $\frac {\partial \mathcal{G}}{\partial U}$ exists in a neighborhood of $(0,u_\e)$ and it is continuous at $(0,u_\e)$.
\item[iv)] The operator $\left(\frac {\partial \mathcal{G}}{\partial U}\right)_{(0,u_\e)}$ is an isomorphism between $ C^{2,\alpha}(\overline\Omega_\e)\cap  C^{0,\alpha}_0 (\overline \Omega_\e)$ and $ C^{0,1}(\overline\Omega_\e)$. 
\end{itemize}
If ${\rm i)-iv)}$ hold by the implicit function theorem we get that,  there exists $\eta_{0,\e}>0$ such that for every $\eta\in (0,\eta_{0,\e})$ there exists a continuous curve $U_\eta$ of solutions to \eqref{general-f-coord-U} in $\Omega_\e$ such that
\beq\label{sti}
\|U_\eta-u_\e\|_{C^{2,\alpha}(\overline\Omega_\e)}\to0\quad\hbox{as }\eta\to0.
\eeq
The function $u_\eta:=\eta^2 U_\eta$ then solves \eqref{general-f-coord} and satisfies
\beq\label{sti-2}
\left\|\frac {u_\eta}{\eta^2}-u_\e\right\|_{C^{2,\alpha}(\overline\Omega_\e)}\to0\quad\hbox{as }\eta\to0,
\eeq
proving \eqref{sstim} and concluding the proof.

Let us verify ${\rm i)-iv)}$. Firsly, i) holds because $\mathcal{G}(0,u_\e)=\Delta _{\R^d}u_\e+1=0$ in $\Omega_\e$ by the definition of $u_\e$.

Next, let us analyze ii).  First we show that $\mathcal{G}(\eta,U)$ is continuous in $(0,u_\e)$.  Assume that $\eta_n\to 0$ and $U_n\to u_\e$ in $C^{2,\alpha}(\overline\Omega_\e)\cap C^{0,\alpha}_0(\overline \Omega_\e)$. 
By the definition of $\mathcal G$ we have,
\[
\mathcal G(\eta_n, U_n)-\mathcal G(0,u_\e)=\Delta _{\R^d}(U_n-u_\e)-\eta_n^2T_{ \eta_n} U_n+ f( \eta_n x, \eta_n^2 U_n)-1.
\]
Observe that 
\[\|\Delta _{\R^d}(U_n-u_\e)\|_{C^{0,\alpha}(\overline\Omega_\e)}\leq C \|U_n-u_\e\|_{C^{2,\alpha}(\overline\Omega_\e)}\to 0 \ \hbox{ as } n\to \infty,\]
and, by the assumptions on $\tilde f$,
 \[|f( \eta_n x, \eta_n^2 U_n)-1|=|f( \eta_n x, \eta_n^2 U_n)-f(0,0)|\leq C\left(|\eta_n| |x| +\eta_n^2 |U_n|\right)\]
where $C$ is the Lipschitz constant of $\tilde f$ in a neighborhood of $(p,0)$. 

Then 
\[\begin{split}
\|f(\eta_n x, \eta_n^2 U_n)-1\|_{C^{0,\alpha}(\overline\Omega_\e)}&\leq C |\eta_n| \  \| x\|_{C^{0,\alpha}(\overline\Omega_\e)}+C\eta_n^2 \|U_n\|_{C^{0,\alpha}(\overline\Omega_\e)}\\
&\leq  C |\eta_n| \ \| x\|_{C^{0,\alpha}(\overline\Omega_\e)}+2C \eta_n^2  \|u_\e\|_{C^{0,\alpha}(\overline\Omega_\e)} \to 0 \ \hbox{ as } n\to \infty,
\end{split}
\]
since 
$\| x\|_{C^{0,\alpha}(\Omega_\e)}\leq (diam (\Omega_\e))^{1-\alpha}$ and $\e$ is fixed, and $\|u_\e\|_{C^{0,\alpha}(\Omega_\e)}\leq C$ by {\rm i)} of Theorem \ref{T2}. 
Further, by the definition of $T_\eta$ in \eqref{stB},  we have that 
\beq\label{4.19}
\begin{split}
\|\eta_n^2 T_{\eta_n} U_n\|_{C^{0,\alpha}(\overline\Omega_\e)}&\leq \eta_n^2
\Big(\|U_n\|_{C^2(\overline\Omega_\e)}\sup_{j,k}\|A_{jk}(\eta_n,x)\|_{C^{0,\alpha}(\overline\Omega_\e)}
+\|U_n\|_{C^{2,\alpha}(\overline\Omega_\e)}\sup_{j,k}\|A_{jk}(\eta_n,x)\|_{C^{0}(\overline\Omega_\e)}
\\
&+\|U_n\|_{C^1(\overline\Omega_\e)}\sup_{k}\|A_{k}(\eta_n,x)\|_{C^{0,\alpha}(\overline\Omega_\e)}+\|U_n\|_{C^{1,\alpha}(\overline\Omega_\e)}\sup_{k}\|A_{k}(\eta_n,x)\|_{C^{0}(\overline\Omega_\e)}\Big)
\end{split}
\eeq
and, by \eqref{stA}, $|A_{jk}(\eta_n,x)|\leq  C|x|^2$ and $|A_{k}(\eta_n,x)|\leq C |x|$. 

Moreover the coefficients $A_{jk}(\eta_n,x)$ have the same regularity in $x$ as the functions $h_{jk}$ in \eqref{metric}, while the coefficients $A_{k}(\eta_n,x)$ have the same regularity in $x$ as the functions $\frac {\partial h_{jk}}{\partial x_i}$. Since the exponential map is a diffeomorphisms then the functions $h_{ij}(x)$ are smooth, and hence the functions $A_{jk}(\eta_n,x)$ and $A_{k}(\eta_n,x)$ are Lipschitz with respect to $x$. This implies that 
\[\begin{split}
\|A_{jk}(\eta_n,x)\|_{C^{0,\alpha}(\overline\Omega_\e)}&\leq C (diam (\Omega_\e))^{1-\alpha}\leq C\\
\|A_{k}(\eta_n,x)\|_{C^{0,\alpha}(\overline\Omega_\e)}&\leq C (diam (\Omega_\e))^{1-\alpha}\leq C
\end{split}\]
since $\e$ is fixed. 
 
 Recalling that $\|U_n\|\leq 2 \|u_\e\|$ then \eqref{4.19} implies that
\[\|\eta_n^2 T_{\eta_n} U_n\|_{C^{0,\alpha}(\overline\Omega_\e)}\to 0 \ \hbox{ as } n\to \infty,\]
which gives the continuity of $\mathcal G(\eta, U)$ in $(0,u_\e)$, when $\eta_n\to 0$. The continuity of $\mathcal G(\eta, U)$  in a neighborhood of $(0,u_\e)$ follows in a very similar manner.

We now consider iii). By the definition of $\mathcal G$, since $T_\eta$ is a linear operator, we have  
\[\left(\frac{\partial \mathcal{G}}{\partial U}\right)(\eta,U)V=
\Delta_{\R^d} V-\eta^2T_\eta V+\eta^2 \frac{\partial f}{\partial u}(\eta x,\eta ^2 U) V \  \hbox{ for }0<|\eta |<\delta_\e.
\]
We extend $\left(\frac{\partial \mathcal{G}}{\partial U}\right)$ in $\eta=0$ letting
\[\left(\frac{\partial \mathcal{G}}{\partial U}\right)(0,U)V=\Delta_{\R^d} V\]
 for every $V\in C^{2,\alpha}(\overline\Omega_\e)\cap C^{0,\alpha}_0(\overline \Omega_\e)$ and for every $U$ in a neighborhood of $u_\e$ in $C^{2,\alpha}(\overline\Omega_\e)\cap C^{0,\alpha}_0(\overline \Omega_\e)$.
We want to show that
\[\left\|\left(\frac{\partial \mathcal{G}}{\partial U}\right)(\eta,U)-\left(\frac{\partial \mathcal{G}}{\partial U}\right)(0,u_\e)\right\|_{\mathcal L(X,Y)}\to 0 \hbox{ as }\eta\to 0,
\]
where $\mathcal L(X,Y)$ denotes the space of linear operator from $X:=C^{2,\alpha}(\overline\Omega_\e)\cap C^{0,\alpha}_0(\overline \Omega_\e)$ into $Y:= C^{0,\alpha}(\overline \Omega_\e)$.

Assume that $\eta_n\to 0$ and $U_n\to u_\e$ in $X$.  We have 
\[\left\|\eta_n^2 \frac{\partial f}{\partial u}(\eta x,\eta_n ^2 U_n) V\right\|_{Y}\leq \eta_n^2\left( c_1 \|V\|_Y+c_2 \|V\|_{C^0(\Omega_\e)}\left(|\eta_n| \ \|x\|_{Y}+\eta_n^2\|U_n\|_{Y}
\right)\right)
\] 
where $c_1=2 \frac{\partial f}{\partial u}(0,0)$ and $c_2$ is the Lipschitz constant of $ \frac{\partial f}{\partial u}$ in a neighborhood of $(0,0)$. This proves that
$\|\eta_n^2 \frac{\partial f}{\partial u}( \eta x,\eta_n ^2 U_n) \|_{\mathcal L(X,Y)}\to 0$ as $n\to \infty$.
Finally, reasoning as in \eqref{4.19}, one can prove that 
\[\|\eta_n^2 T_{ \eta_n} \|_{\mathcal L(X,Y)}\to 0 \hbox{ as } n\to \infty,\]
concluding the proof of the continuity of $\left(\frac{\partial \mathcal{G}}{\partial U}\right)$ in $(0,u_\e)$.

To conclude, note that iv) holds since
\[\left(\frac{\partial \mathcal{G}}{\partial U}\right)(0,u_\e)V=\Delta_{\R^d} V,
\]
which is an isomorphism between $ C^{2,\alpha}(\overline \Omega_\e)\cap  C^{0,\alpha}_0 (\overline \Omega_\e)$ and $ C^{0,\alpha}(\overline\Omega_\e)$.

Therefore, the implicit function theorem applies to the operator $\mathcal{G}$ defined in \eqref{g-epsilon} and from $\mathcal{G}(0,u_\e)=0$ we get that
there exists $\eta_{0,\e}>0$ such that for every $\eta\in (-\eta_{0,\e},\eta_{0,\e})$ there exists a continuous curve $U_\eta$ such that
$$\mathcal{G}(\eta,U_\eta)=0.$$
Hence, for any $\eta\in (0,\eta_{0,\e})$
\beq
\tilde u_\eta(q)=\eta^2U_\eta\big(\Phi_\eta^{-1}(q)\big)=\eta^2U_\eta\big(\exp^{-1}(\eta q)\big)
\eeq
is a solution to \eqref{general-f} satisfying (using \eqref{sti})
\beq
\left|\left|\frac{\tilde u_\eta\big(\exp_p(\eta -)\big)}{\eta^2}-u_\e
\right|\right|_{C^{2,\alpha}(\overline \Omega_\e)}\to0\quad\hbox{as }\eta\to0,
\eeq
which ends the proof.
\end{proof}

An immediate consequence of the previous proposition is the following:

\begin{corollary}\label{corg}
For any integer $n$ fixed, $n\geq 2$,  for every $\e$ fixed, that verifies $0<\e<\e_0$, there exists $\eta_{1,\e}>0$ such that for any $0<\eta<\eta_{1,\e}$
the problem \eqref{general-f} admits a solution $\tilde u_\eta$ that has exactly $n$ nondegenerate local maximum points and $n-1$ nondegenerate saddle points in $\widetilde\Omega_\eta$.
\end{corollary}
\begin{proof}
By Theorem \ref{T2}  the solution $u_\e$ to \eqref{tor} has exactly $n$ local, nondegenerate maximum points and $n-1$ nondegenerate saddle points in $\Omega_\e$.  By \eqref{sti-2} then there exists $\hat \eta_{1,\e}>0$ such that, for every $0<\eta<\hat \eta_{1,\e}$ the function $\frac 1{\eta^2}u_\eta$ has 
exactly $n$ maximum points and $n-1$  saddle points in $\Omega_\e$ and they are all nondegenerate.  The same is true, of course, for the function $u_\eta$.  Letting $\bar x$ be one of the critical points of $u_\eta$ and $\bar q=\Phi_\eta(\bar x)$ then by \eqref{metric} and \eqref{metric-inv} we have
\[\begin{split}
\frac{\partial  \tilde u_\eta }{\partial x_k}(\bar q)&=\sum _{j=1}^d
g^{kj}\frac{\partial u_\eta }{\partial x_j}(\bar x)=0,
\\
\frac{\partial ^2 \tilde u_\eta }{\partial x_k\partial x_h}(\bar q)&=\sum _{i,j=1}^d
g^{kj}g^{hi}\frac{\partial ^2 u_\eta }{\partial x_j\partial x_i}(\bar x)=\frac 1{\eta^4}\frac{\partial ^2 u_\eta }{\partial x_k\partial x_h}(\bar x) (1+o(1))
\end{split}
\]
as $\eta\to 0$. This implies that, all the critical points of $u_\eta$ provide critical points of $\tilde u_\eta$ and viceversa, since the matrix $g^{kj}$ is invertible.
Then there exists $\eta_{1,\e}>0$ such that, for any $0<\eta<\eta_{1,\e}$,  we have $\det \nabla^2 \tilde u_\eta (\bar q)\neq 0$, so that all the critical points of $\tilde u_\eta$ are nondegenerate. This concludes the proof.
\end{proof}

\section{Proof of Theorem \ref{T1}}\label{s5}
We start recalling some definitions:
\begin{definition}\label{def-geo-star}
A set $U\subset \mathcal M$ is {\em geodesically starshaped} with respect to a point~$q\in U$ if for every $z\in U$ there exists a unique minimizing geodesic connecting $q$ and~$z$ contained in $U$. Equivalently, there exists some set $V\subset \R^d$, starshaped with respect to the origin, such that $U=\exp_q(V)$. 
\end{definition}

\begin{definition}\label{def-conv}
A set $U\subset \mathcal M$ is {\em geodesically convex}\/ if for every $q,z\in U$ there exists a connecting minimizing geodesic contained in $U$.
 Equivalently, $U$ is geodesically starshaped with respect to all its points, so for every  $q\in U$ there is a starshaped set $V_q\subset\R^d$ such that  $U=\exp_q(V_q)$.
\end{definition}

Using these definition, it is easy to make precise the idea that a set (or rather a sequence of sets) is ``almost'' geodesically convex:
\begin{definition}\label{alcon}
A sequence of sets $U_k\subset \mathcal M$ is {\em almost geodesically convex}\/ if
\beq\label{z}
\frac{\mea_g \big(U_k\setminus \Sta_g (U_k)\big)}{\mea_g \big(U_k\big)}\to0\quad\hbox{as }k\to+\infty,
\eeq
where 
$\Sta_g (U_k)$ denotes the set of points of $U_k$ with respect to which $U_k$ is geodesically starshaped.
\end{definition}

We are in position to prove our main result:
\begin{proof}[Proof of Theorem \ref{T1}]
Take a sequence $\e_k$ such that $0<\e_k<\e_0$, with $\e_0$ as in Theorem \ref{T2},  and $\e_k\to 0$, as $k \to \infty$.  Corresponding to every $\e_k$ fixed,  there exists, by Proposition \ref{prog} and Corollary \ref{corg},  $\eta_{1,\e_k}>0$ such that, for every $\eta \in (0,\eta_{1,\e_k})$ problem \eqref{general-f} admits a solution $u_\eta$ in $\widetilde \Omega_\eta:=\Phi_{\eta}(\Omega_{\e_k})$, by \eqref{omt},  which has exactly $n$ nondegenerate maximum points and $n-1$ nondegenerate saddle points. For every $k$, we choose a value $\eta_k\in (0,\eta_{1,\e_k})$ such that $\widetilde \Omega_{\eta_k}=\Phi_{\eta_k}(\Omega_{\e_k})$ and 
$\tilde u_k$ is the solution to \eqref{general-f} in $\widetilde \Omega_{\eta_k}$.  As said before the solutions $\tilde u_k$ have exactly $n$ nondegenerate maximum points in $\widetilde \Omega_{\eta_k}$ and $n-1$ nondegenerate saddle points in $\widetilde \Omega_{\eta_k}$, proving a).

Since the domains $\Omega_{\e_k}$ are smooth by Theorem~\ref{T2} then  the domains $\widetilde \Omega_{\eta_k}$ are smooth too.  Moreover the domains $\Omega_{\e_k}$ are all starshaped with respect to the origin by Theorem~\ref{T2}, and this implies that the sequence of domains $\widetilde \Omega_{\eta_k}=\Phi_{\eta_k}\left(\Omega_{\e_k}\right)$ are geodesically starshaped  with respect to $p$, by Definition \ref{def-geo-star}. This proves {\rm b)}.

The proof of point $c)$ is slightly more delicate and requires a few more details. We aim to show that the set $\eta_k D_{\e_k}$, with $D_{\e_k}$ defined in \eqref{d-epsilon} is contained in $\Sta_g(\widetilde \Om_{\eta_k})$, provided that $\e_k,\eta_k$ are suitably small. From Definitions \ref{def-geo-star} and \ref{def-conv} a point $q\in \widetilde \Om_{\eta_k}$ belongs to $\Sta_g(\widetilde\Om_{\eta_k})$ if there exists a set $V_q\subset \R^d$, starshaped with respect to the origin, such that 
\[\widetilde \Om_{\eta_k}=\exp_q(V_q) \ \text{or, equivalently, } \ V_q=\exp_q^{-1}(\widetilde \Om_{\eta_k}).\]
Moreover, since $q\in \widetilde \Om_{\eta_k}$,  we have $q=\exp_p(\eta_k \xi)$ for some $\xi\in \Om_{\e_k}$. This naturally leads us to introduce the map
\begin{equation}\label{fe}
\psi_{k,\xi}:=\exp_{\exp_p(\eta_k\xi)}^{-1}\circ\exp_p
\end{equation}
that we will apply to points in $\eta_k \Om_{\e_k}$, 
and define the set
\begin{equation}
\Om_{k,\xi}:= \frac 1{\eta_k}\psi_{k,\xi}(\eta_k \Om _{\e_k}).
\end{equation}

By picking a decreasing sequence of small constants $0<\eta_{k+1}<\eta_k$, we can assume that the closure of~ $\eta_k\Om_{\e_k}$ is contained in a small disk  $S\subset \R^d$  for all $k$. Hence the map $\psi_{k,\xi}$ is well defined on $S$ for all~$k$.
 Furthermore, since $\psi_{k,\xi}$ is a smooth function of~$\xi$ and satisfies $\psi_{k,0}=I$, it is clear that $\psi_{k,\xi}$ is close to the identity in the sense that 
\[
\sup_{\xi\in \Om_{\e_k}}\|\psi_{k,\xi}-I\|_{C^2(S)}\le C\eta_k\,.
\]
More precisely,
\begin{equation}\label{psi-svil}
\psi_{k,\xi}(x)=x-\eta_k\xi + \psi''_{k,\xi}(x)\,,
\end{equation}
where the second-order correction is bounded as $\|\psi''_{k,\xi}\|_{C^2(S)}\le C\eta_k^2$. The constant~$C$ depends on~$\e_k$ but not on the small parameter $\eta_k$.

As a consequence of this, given any point $\xi\in\Om_{\e_k}$, one can write 
\[
\widetilde\Om_{\eta_k}=\exp_{\exp_p(\eta_k\xi)}(\eta_k \Om_{k,\xi})\,,
\]
where $\Om_{k,0}=\Om_{\e_k}$ and, in general, by \eqref{psi-svil}
\begin{equation}\label{E.translation}
\Om_{k,\xi}=\frac1{\eta_k}\psi_{k,\xi}(\eta_k\Om_{\e_k})=\Om_{\e_k}-\xi+ \frac{\psi''_{k,\xi}(\eta_k\Om_{\e_k})}{\eta_k}.
\end{equation}
We can now define the map 
\begin{equation}\label{psi-k}
\Psi_{k,\xi}(x):=x+ \eta_k^{-1}\psi''_{k,\xi}(\eta_kx)
\end{equation}
which is well defined on $\frac S{\eta_k} $ and 
is then close to the identity in~$C^2(\frac S {\eta_k} )$ because
\[
\|\Psi_{k,\xi}-I\|_{C^2(\frac S {\eta_k} )}=\frac1{\eta_k}\|\psi''_{k,\xi}\|_{C^2(S)}<C\eta_k,
\]
where the constant $C$ depends on $\e_k$ but not on $\eta_k$. Then
\[
\sup_{\xi\in \Om_{\e_k}}\|\Psi_{k,\xi}-I\|_{C^2(\frac S{\eta_k})}<C\eta_k.
\]
Also, note that the large balls $S/\eta_k$ satisfy
\[
\frac S{\eta_1}\subset \frac S{\eta_2}\subset\cdots 
\]
because the sequence $\eta_k$ is decreasing.

Since $\Om_{k,\xi}$ only differs from $\Psi_{k,\xi}(\Om_{\e_k})$ by a translation by~$\xi\in\Om_{\e_k}$ as a consequence of~\eqref{E.translation}, we conclude from Corollary~\ref{C.perturbation} that the set $D_{\e_k}$, as defined in Proposition~\ref{lem1}, satisfies
\[
D_{\e_k}-\xi\subset \Sta(\Om_{k,\xi})
\]
provided that $\eta_k$ is small enough (depending on~$\e_k$). In this case, $\Om_{k,\xi}\subset\R^d$ is starshaped with respect to the origin whenever $\xi\in D_{\e_k}$, which implies 
\[
\widetilde D_k:=\exp_p(\eta_k D_{\e_k})\subset\Sta_g(\widetilde\Om_{\eta_k})\,.
\]
To prove c), using Definition \ref{alcon}, we then only have to compute
\begin{align*}
\frac{\mea_g \big(\widetilde\Omega_{\eta_k}\setminus \Sta_g (\widetilde\Omega_{\eta_k})\big)}{\mea_g \big(\widetilde\Omega_{\eta_k}\big)}&\le \frac{\mea_g \big(\widetilde\Omega_{\eta_k}\setminus \widetilde D_k\big)}{\mea_g \big(\widetilde\Omega_{\eta_k}\big)}
=\frac{\int _{\Omega_{\e_k}\setminus D_{\e_k}}\rho_k(\eta_kx) dx}{\int _{\Omega_{\e_k}}\rho_k(\eta_kx) dx}\,,\end{align*}
where $dV_g=:\eta_k^d \rho_k(\eta_kx) dx$ is the volume form written in rescaled normal coordinates. Since $|\rho_k(\eta_kx)-1|\le C\eta_k$, for $k$ large enough, 
\begin{align*}
\frac{\mea_g \big(\widetilde\Omega_{\eta_k}\setminus \Sta_g (\widetilde\Omega_{\eta_k})\big)}{\mea_g \big(\widetilde\Omega_{\eta_k}\big)}&\le \frac{(1+C\eta_k)\int _{\Omega_{\e_k}\setminus D_{\e_k}
}dy}{(1-C\eta_k)\int _{\Omega_{\e_k}}dy}\\
&\le
(1+3 C\eta_k)\frac{\mea \big(\Omega_{\e_k}\setminus D_{\e_k}\big)}{\mea \ (\Omega_{\e_k})}\to0
\end{align*}
as $k\to \infty$ by \eqref{finale} of Lemma \ref{lem2}.  Thus,  claim $c)$ follows.
\end{proof}

\noindent\textit{Acknowledgments.}
This work has been developed within the framework of the project e.INS- Ecosystem of Innovation for
Next Generation Sardinia (cod. ECS 00000038) funded by the Italian Ministry for Research and Education
(MUR) under the National Recovery and Resilience Plan (NRRP) - MISSION 4 COMPONENT 2, "From
research to business" INVESTMENT 1.5, "Creation and strengthening of Ecosystems of innovation" and
construction of "Territorial R\&D Leaders".
The second author is funded by Next Generation EU-
CUP J55F21004240001, “DM 737/2021 risorse 2022–2023”. The last two authors are partially funded by Gruppo Nazionale per l’Analisi Matematica, la Probabilità e le loro Applicazioni (GNAMPA) of the Istituto Nazionale di Alta Matematica (INdAM). This work has received funding from the European Research Council (ERC) under the European Union's Horizon 2020 research and innovation programme through the grant agreement~862342 (A.E.). A.E.\ is partially supported by the grants CEX2023-001347-S, RED2022-134301-T and PID2022-136795NB-I00 funded by MCIN/AEI.

\bibliography{EncisoGladialiGrossiV4.bib}
\bibliographystyle{abbrv}

\end{document}